\documentclass[12pt]{amsart}
\usepackage{amsmath,amssymb,amsthm}
\newcommand{\re}{\mathbb{R}}
\newcommand{\co}{\mathbb{C}}

\newcommand{\na}{\mathbb{N}}

\newcommand{\cc}{\mathcal{C}}
\newcommand{\SSS}{\mathcal{S}}
\newcommand{\z}{\bar z}
\newcommand{\w}{\bar w}
\newcommand{\rp}{\mbox{Re}}
\newcommand{\ip}{\mbox{Im}}
\newcommand{\srp}{\mbox{\rm \scriptsize{Re}}}
\newcommand{\sip}{\mbox{\rm \scriptsize{Im}}}

\long\def\symbolfootnote[#1]#2{\begingroup%
\def\thefootnote{\fnsymbol{footnote}}\footnote[#1]{#2}\endgroup}
\newcommand{\boxx}{\rule{2.12mm}{3.43mm}}

\numberwithin{equation}{section}

\title[Continuous solutions of Cauchy-Riemann equations]{Continuous
  solutions of nonlinear Cauchy-Riemann equations and
  pseudoholomorphic curves in normal coordinates}

\author[Adam Coffman]{Adam Coffman}
\address{Department of Mathematical Sciences \\ Indiana University -
  Purdue University Fort Wayne \\ 2101 E.\ Coliseum Blvd. \\ Fort
  Wayne, IN, USA 46805-1499}
\email{CoffmanA@ipfw.edu}
\urladdr{http://www.ipfw.edu/math/}

\author[Yifei Pan]{Yifei Pan}
\address{College of Mathematics and Information Sciences \\ Jiangxi Normal University \\ Nanchang, P.R.China}
\email{Pan@ipfw.edu}

\author[Yuan Zhang]{Yuan Zhang}
\address{Department of Mathematical Sciences \\ Indiana University -
  Purdue University Fort Wayne \\ 2101 E.\ Coliseum Blvd. \\ Fort
  Wayne, IN, USA 46805-1499}
\email{ZhangYu@ipfw.edu}

\newtheorem{thm}{Theorem}[section]
\newtheorem{prop}[thm]{Proposition}
\newtheorem{lem}[thm]{Lemma}
\newtheorem{cor}[thm]{Corollary}

\theoremstyle{definition}

\newtheorem{notation}[thm]{Notation}
\newtheorem{example}[thm]{Example}

\theoremstyle{remark}

\newtheorem{rem}[thm]{Remark}

\begin{document}

\begin{abstract}
  We establish elliptic regularity for nonlinear, inhomogeneous
  Cauchy-Riemann equations under weak assumptions, and give a
  counterexample in a borderline case.  In some cases where the
  inhomogeneous term has a separable factorization, the solution set
  can be explicitly calculated.  The methods also give local
  parametric formulas for pseudoholomorphic curves with respect to
  some continuous almost complex structures.
\end{abstract}

\maketitle

\section{Introduction}\symbolfootnote[0]{MSC 2010: 35J46; 30G20, 32Q65.}

We consider the nonlinear, inhomogeneous Cauchy-Riemann equation: for
open sets $\Omega_1,\Omega_2\subseteq\co$ and a function
$u:\Omega_1\to\Omega_2$, the equation is
\begin{equation}\label{eq29}
 \frac{\partial u}{\partial\z}=E(z,u).
\end{equation}

Section \ref{sec2} starts with the linear case, Theorem \ref{lem4.3},
establishing some regularity of solutions $u$ under minimal
assumptions: $u$ is continuous, the partial derivatives $u_x$ and
$u_y$ (and the LHS $\frac12(u_x+iu_y)$) exist except possibly on some
small set, and the linear equation $\frac{\partial
u}{\partial\z}=P(z)$ holds almost everywhere for $P\in L^p_{loc}$,
$p\ge2$.  An analogue in the homogeneous case is the Looman-Menchoff
Theorem, that a continuous, but not necessarily $\cc^1$, function with
zero $\z$-derivative must be analytic.  Regularity of $u$ satisfying
the nonlinear equation (\ref{eq29}) then follows in some corollaries
of Theorem \ref{lem4.3}.  In Section \ref{sec4} we give a new example
of a differentiable function $u$ satisfying $\partial
u/\partial\z=P(z)$, where $P$ is continuous on $\co$ but $\partial
u/\partial z$ is not.

In Section \ref{sec1}, we consider the ``separable'' case of the
nonlinear Cauchy-Riemann equation where the RHS of (\ref{eq29})
factors in the form $E(z,u)=f(u)g(z)$ with $f$ holomorphic.  We state
a local existence result in a special case (Theorem \ref{lem4.6}),
but our main goal in Section \ref{sec1} is to explicitly compute local
formulas for solutions $u$ without strong {\sl a priori} assumptions
on the regularity of $u$.

In Section \ref{sec3}, we apply the results of Sections \ref{sec2} and
\ref{sec1} to find formulas for all the $J$-holomorphic curves in
certain coordinate charts in some almost complex $4$-manifolds.
Example \ref{ex4.3} uses the counterexample from Section \ref{sec4} to
show that a continuous almost complex structure can admit a
$J$-holomorphic curve which is differentiable but not $\cc^1$.

\section{Nonlinear Cauchy-Riemann equations}\label{sec2}

\begin{notation}
  For $z=x+iy\in\Omega\subseteq\co=\re^2$, and a function
  $u:\Omega\to\co$, $u_x=\frac{\partial u}{\partial x}$ and
  $u_y=\frac{\partial u}{\partial y}$ are the complex valued pointwise
  partial derivatives with respect to the real coordinates.  If both
  $u_x$ and $u_y$ exist at a point, then $u_{z}=\frac{\partial
    u}{\partial z}=\frac12(u_x-iu_y)$ and $u_{\z}=\frac{\partial
    u}{\partial\z}=\frac12(u_x+iu_y)$ are the pointwise $z$- and
  $\z$-derivatives.  The distributional $\z$-derivative of $u$ on
  $\Omega$ (and similarly for $z$) is the operator, denoted by
  $\partial_{\z} u$, which maps compactly supported smooth test
  functions $\varphi\in\cc^{\infty}_0(\Omega)$ to
  $-\int_{\Omega}u\frac{\partial\varphi}{\partial\z}$.  We say that
  $\partial_{\z}u$ is represented on $\Omega$ by a function $r$ to
  mean that
  $-\int_{\Omega}u\frac{\partial\varphi}{\partial\z}=\int_{\Omega}r\varphi$.
\end{notation}
A distributional derivative represented by $r$ on a domain behaves as
expected under restriction: if $\Omega_2$ is an open subset of
$\Omega_1$, and $\partial_{\z}u$ is represented on $\Omega_1$ by $r$,
then $\partial_{\z}(u|_{\Omega_2})$ is represented on $\Omega_2$ by
$r|_{\Omega_2}$.
\begin{notation}
  Let $R\Subset\Omega$ denote that $R$ is a bounded, open rectangle of
  the form $(a_1,b_1)\times(a_2,b_2)$, with closure $\overline{R}$
  contained in the open set $\Omega\subseteq\co$.  Let $\partial R$
  denote the boundary of $R$.
\end{notation}
Usually, Green's Theorem is stated with a $\cc^1$ or $W^{1,1}$
hypothesis (\cite{aim} Theorem 2.9.1).  However, in a situation where
the partial derivatives exist but may not all be integrable, the
following version of Green's Theorem due to Cohen (\cite{cohenthesis},
\cite{cohen}, \cite{cv}, \cite{gm} Theorem 8) applies.

\begin{prop}\label{lem4.2}
  Suppose $v:\Omega\to\co$ is continuous and satisfies the following
  condition:
  \begin{itemize}
    \item[$(*)$] The partial derivatives $v_x$, $v_y$ exist at every
  point in $\Omega$ except for countably many.
  \end{itemize}
  Then, for any $R\Subset\Omega$, if $\frac{\partial v}{\partial\z}\in
  L^1(R)$, then
  $$\int_{\partial R}v(z)dz=2i\int_R\frac{\partial
  v}{\partial\z}dxdy.$$ \boxx
\end{prop}
\begin{rem} 
  The statement of Proposition \ref{lem4.2} can be generalized to
  shapes other than rectangles, and the condition $(*)$ can be
  weakened to allow a larger exceptional set: see \cite{cv}.  The
  property $(*)$ can also be assumed to hold only on one particular
  rectangle $R$, but the above formulation is more convenient for us,
  and as a practical matter, the condition $(*)$ on the classical
  derivatives is something more easily checked than properties of
  distributional derivatives.  The main significance of the
  Proposition is that its hypothesis omits any assumption about the
  integrability or continuity of the individual partial derivatives
  $v_x$, $v_y$, or $v_z$.  We also remark that the integrand on the
  RHS is the pointwise derivative (where it exists), not the
  distributional derivative.
\end{rem}
  Cohen's proof was motivated by the earlier Looman-Menchoff Theorem,
  which we recall here from (\cite{n}, \cite{gm} Theorem 11) as a
  Proposition, to be used in Section \ref{sec1}.
\begin{prop}\label{prop2.4}
  Suppose $v:\Omega\to\co$ is continuous and satisfies condition
  {\rm{$(*)$}}.  If \begin{equation}\label{eq94} \frac{\partial
  v}{\partial\z}=0
  \end{equation}
  almost everywhere in $\Omega$, then $v$ is holomorphic on $\Omega$.
  \boxx
\end{prop}
The following Theorem considers an inhomogeneous, linear version of
(\ref{eq94}).  In the following Proof, some steps are similar to steps
in \cite{bbc} \S2 and \cite{cv}, and the last two paragraphs recall
well-known regularity methods, but we give enough details to show
exactly where Proposition \ref{lem4.2} is used to establish the
necessary integration by parts.

\begin{thm}\label{lem4.3}
  Suppose $u:\Omega_1\to\Omega_2$ is continuous, satisfies
  {\rm{$(*)$}}, and there is a function $P:\Omega_1\to\co$ so that
  $P\in L^p_{loc}(\Omega_1)$ for some $p$, $2\le p<\infty$, and
  $$\frac{\partial u}{\partial\z}=P(z)$$ almost everywhere.  Then, for
  any $R\Subset\Omega_1$, $u|_R\in W^{1,2}(R)$.  If, further, $p>2$,
  then for $\alpha=1-\frac2p$, $u|_R\in\cc^{0,\alpha}(R)$.
\end{thm}
\begin{proof}
  The restriction $u|_R$ is continuous and bounded on $R$, and an
  element of $L^2(R)$.  The following argument uses the assumption on
  the classical pointwise derivatives to draw this conclusion about
  the distributional derivatives: $u|_R\in W^{1,2}(R)$, meaning that
  its distributional derivatives on $R$, $\partial_{\z}(u|_R)$ and
  $\partial_z(u|_R)$, are represented by functions in $L^2(R)$.

  By compactness, there is a larger rectangle with $R\Subset
  R_1\Subset\Omega_1$.  Let $u_1$ and $P_1$ be the restrictions of $u$
  and $P$ to $R_1$, so $P_1:R_1\to\co$ satisfies
  \begin{equation}\label{eq87}
    \left.\frac{\partial u}{\partial\z}\right|_{R_1}=_{a.e.}P_1\in
    L^p(R_1)\subseteq L^2(R_1).
  \end{equation}
  For a test function $\varphi\in\cc_0^\infty(R_1)$, the
  product $u_1\varphi$ satisfies, for all $z$ except in some set of
  measure $0$ (which includes the exceptional set from $(*)$),
  \begin{eqnarray}
    \frac{\partial}{\partial\z}(u_1(z)\varphi(z))&=&\left(\frac{\partial}{\partial\z}u_1\right)\varphi(z)+u_1(z)\left(\frac{\partial}{\partial\z}\varphi\right)\nonumber\\
    &=_{a.e.}&P_1(z)\varphi(z)+u_1(z)\left(\frac{\partial}{\partial\z}\varphi\right).\label{eq23}
  \end{eqnarray} 
  We emphasize that Equations (\ref{eq87}) and (\ref{eq23}) are a.e.\
  equalities of functions, not equalities of distributions.  $P_1$ and
  the RHS of (\ref{eq23}) are defined for all $z\in R_1$, while
  $u_{\z}$ and the LHS of (\ref{eq23}) may be undefined for some $z$
  in a set of measure $0$.  Because the two functions differ only on a
  set of measure $0$ and
  $P_1\varphi+u_1\frac{\partial\varphi}{\partial\z}\in
  L^p(R_1)\subseteq L^1(R_1)$,
  $\frac{\partial}{\partial\z}(u_1\varphi)$ is also in $L^1(R_1)$.

  For $p\ge2$, define this function $P_2:\co\to\co$:
  $$P_2(z)=\left\{\begin{array}{ll}P_1(z)&z\in R_1\\0&z\notin
  R_1\end{array}\right.,$$ so $P_2\in L^p(\co)\cap L^2(\co)$.  The
  Cauchy transform $\cc(P_2)$ is an element of $L^1_{loc}(\co)$
  (\cite{aim} Theorem 4.3.9, Theorem 4.3.13), and its distributional
  derivative on $\co$, $\partial_{\z}\cc(P_2)$, is represented by
  $P_2\in L^2(\co)$ (\cite{aim} Theorem 4.3.10).  The restriction
  $\cc(P_2)|_{R_1}$ has distributional derivative
  $\partial_{\z}((\cc(P_2)|_{R_1}))$ on $R_1$ represented by
  $P_2|_{R_1}=P_1$.  The restriction $u_1-(\cc(P_2)|_{R_1})$ is
  integrable on $R_1$, and the distributional derivative on $R_1$
  satisfies, for $\varphi\in\cc_0^\infty(R_1)$,
  \begin{eqnarray}
    \partial_{\z}(u_1-(\cc(P_2)|_{R_1})):\varphi&\mapsto&-\int_{R_1}(u_1-(\cc(P_2)|_{R_1}))\frac{\partial\varphi}{\partial\z}\nonumber\\
    &=&-\int_{R_1}u_1\frac{\partial\varphi}{\partial\z}-\int_{R_1}P_1\varphi\label{eq26}\\
    &=&-\int_{R_1}\left(P_1\varphi+u_1\frac{\partial\varphi}{\partial\z}\right)\nonumber\\
    &=&-\int_{R_2}\frac{\partial}{\partial\z}\left(u_1\varphi\right)\label{eq24}\\
    &=&0.\label{eq25}
  \end{eqnarray}
  Line (\ref{eq24}) follows from Equation (\ref{eq23}), and
  $R_2\Subset R_1$ is a smaller rectangle with interior containing the
  support of $\varphi$.  Line (\ref{eq25}) uses Proposition
  \ref{lem4.2}, and this is the key technical step using the
  assumptions on the $\z$-derivative without any integrability of the
  $z$-derivative.  It follows from (\ref{eq26}) and (\ref{eq25}) that
  the distributional derivative on $R_1$,
  $\partial_{\z}u_1=\partial_{\z}\left(\cc(P_2)|_{R_1}\right)$, is
  represented by $P_1$, which is a.e.\ equal to $\frac{\partial
  u_1}{\partial\z}$ as in (\ref{eq87}), so the distributional and
  a.e.\ pointwise $\z$-derivatives coincide.  It follows by
  restriction that $u|_R=u_1|_R$ has distributional derivative on $R$,
  $\partial_{\z}(u|_R)$, represented by $P_1|_R=P|_R\in L^2(R)$.

  Also, Weyl's Lemma (\cite{aim} Lemma A.6.10, \cite{gm} Theorem 9)
  applies, so there exists a holomorphic function $\Phi:R_1\to\co$
  equal to $u_1-(\cc(P_2)|_{R_1})$ as an element of $L^1(R_1)$.  The
  Beurling transform, $\SSS(P_2)\in L^2(\co)$, is a function defined
  almost everywhere in $\co$ (\cite{aim} Theorem 4.0.10) that
  represents the distributional derivative of $\cc(P_2)$ on $\co$,
  $\partial_z(\cc(P_2))$ (\cite{aim} Theorem 4.3.10).  So, the
  distributional derivative of $u_1$ on $R_1$,
  $\partial_z(u_1)=\partial_z(\Phi+(\cc(P_2)|_{R_1}))$, is represented
  by $\frac{\partial\Phi}{\partial z}+\left(\SSS(P_2)|_{R_1}\right)$.
  The restrictions $\Phi|_R$ and $\SSS(P_2)|_R$ are both in $L^2(R)$,
  so the distributional derivative of $u$ on $R$, $\partial_z(u|_R)$,
  is represented by $\Phi|_R+\left(\SSS(P_2)|_R\right)\in L^2(R)$.

  For $p>2$, $\cc(P_2)\in\cc^{0,\alpha}(\co)$ (\cite{aim} Theorem
  4.3.13), and the restriction $\Phi|_R$ is in $\cc^{0,\alpha}(R)$, so
  by continuity, $u|_R=\Phi|_R+(\cc(P_2)|_{R})$ pointwise everywhere
  in $R$ and $u|_R\in\cc^{0,\alpha}(R)$.
\end{proof}

\begin{cor}\label{thm4.2}
  Let $E:\Omega_1\times\Omega_2\to\co$, let $u:\Omega_1\to\Omega_2$ be
  continuous, and suppose that $u$ satisfies {\rm{$(*)$}}, and
  \begin{equation}\label{eq27}
    \frac{\partial u}{\partial\z}=E(z,u(z))
  \end{equation}
  almost everywhere.
  \begin{itemize}
    \item If $E$ is continuous, then for any $R\Subset\Omega_1$,
    $u|_R\in\cc^{0,\alpha}(R)$ for all $0<\alpha<1$.
    \item If $0<\beta<1$ and
    $E\in\cc^{0,\beta}_{loc}(\Omega_1\times\Omega_2)$, then for any
    $R\Subset\Omega_1$, $u|_R\in\cc^{1,\beta}(R)$.
    \item For $r\in\na$, $r=\infty$, or $r=\omega$, if
    $E\in\cc^r(\Omega_1\times\Omega_2)$, then $u\in\cc^r(\Omega_1)$.
  \end{itemize}
\end{cor}
\begin{proof}
  First, if $E$ is continuous on $\Omega_1\times\Omega_2$, then for
  any $p\ge2$, $u$ satisfies the hypotheses of Theorem \ref{lem4.3},
  with $\frac{\partial u}{\partial\z}$ equal almost everywhere to the
  continuous function $P(z)=E(z,u(z))\in L^p_{loc}(\Omega_1)$.  The
  conclusion from the Theorem is that for any $R\Subset\Omega_1$ and any
  $0<\alpha<1$, 
  \begin{equation}\label{eq72}
    u|_R\in W^{1,2}(R)\cap\cc^{0,\alpha}(R).
  \end{equation}

  For the second claim of the Corollary, consider larger rectangles
  $R\Subset R_2\Subset R_1\Subset\Omega_1$.  $E(z,w)$ is
  $\cc^{0,\beta}$ on the compact product $\overline{R_1}\times
  u(\overline{R_1})$, and the composite $E(z,u(z))$ is continuous,
  with H\"older exponent $\alpha\beta$.  Because the RHS of
  (\ref{eq27}), restricted to $z\in R_1$, is in
  $\cc^{0,{\alpha\beta}}(R_1)$, it follows from (\ref{eq72}) and
  \cite{aim} Theorem 15.0.7 that
  $u|_{R_1}\in\cc^{1,\alpha\beta}_{loc}(R_1)$.  The composite
  $E(z,u(z))$ is now in $\cc^{0,{\beta}}(R_2)$, and \cite{aim} Theorem
  15.0.7 applies again to establish the claim.

  For the third claim with $r=1$, because the conclusion is a local
  property of $u$, it is enough to work with the same rectangle $R$ as
  the previous case and $u|_R$ as in (\ref{eq72}).  If
  $E\in\cc^1(\Omega_1\times\Omega_2)$, then the composite $E(z,u(z))$
  is $\cc^{0,\alpha}$ on $R$, and again by \cite{aim} Theorem 15.0.7,
  $u|_R\in\cc^{1,\alpha}_{loc}(R)$.  So, $u\in\cc^1(\Omega_1)$.  If
  $E\in\cc^2(\Omega_1\times\Omega_2)$, then the composite
  $E(z,u(z))\in\cc^{1,\alpha}_{loc}(R)$, so
  $u|_R\in\cc^{2,\alpha}_{loc}(R)$.  For $r>1$, the bootstrap
  technique applies, iterating $r$ times when $E$ is $\cc^r$, and if
  $E$ is smooth, then $u$ is smooth.

  When $E(z,w)$ is real analytic, $u$ is smooth, and using the chain
  rule (\cite{aim} \S2.9.1) gives:
  \begin{eqnarray}
    \Delta(u)&=&4\frac{\partial}{\partial z}\frac{\partial
    u}{\partial\z}=4\frac{\partial}{\partial z}(E(z,u(z))\nonumber\\
    &=&4(E_z(z,u(z))+E_w(z,u(z))\frac{\partial u}{\partial
    z}+E_{\w}(z,u(z))\overline{\frac{\partial u}{\partial\z}}).\nonumber
  \end{eqnarray}
  This complex equation (or the system of two real equations
  $\Delta(\rp(u))=\rp(\Delta(u))$ and $\Delta(\ip(u))=\ip(\Delta(u))$)
  is a second order, nonlinear, elliptic system where the RHS is a
  real analytic expression in $z$, $u$, (or their real and imaginary
  parts) and the first derivatives of $u$.  For such a system, $\cc^3$
  solutions $u$ must be real analytic (\cite{m}).
\end{proof}
\begin{cor}\label{cor4.7}
  For a connected open set $\Omega_1\subseteq\co$, suppose
  $u:\Omega_1\to\Omega_2$ is continuous and satisfies {\rm{$(*)$}}.
  Given $w_0\in\Omega_2$, let $Z_0=\{z\in\Omega_1:u(z)=w_0\}$.  If
  there is a function $A:\Omega_1\setminus Z_0\to\co$ so that $A\in
  L^p(\Omega_1\setminus Z_0)$ for some $p>2$ and
  \begin{equation}\label{eq28}
    \frac{\partial u}{\partial\z}=(u(z)-w_0)A(z)
  \end{equation}
  almost everywhere in $\Omega_1\setminus Z_0$, then either $Z_0$ is a
  set of isolated points in $\Omega_1$ or $Z_0=\Omega_1$.
\end{cor}
\begin{rem}
  Proposition C of \cite{gr} is similar to the above statement, but
  its hypothesis includes the distributional derivative version of
  (\ref{eq28}) (see also \cite{is1}, \cite{is}).  In view of the Proof
  of Theorem \ref{lem4.3}, the distributional derivative equation is
  equivalent to the a.e.\ pointwise property under these conditions,
  so Corollary \ref{cor4.7} is also a corollary of \cite{gr}
  Proposition C.  Because we need formula (\ref{eq71}) in the Proof of
  Theorem \ref{thm2.9}, here we sketch a Proof of Corollary
  \ref{cor4.7} using the same methods as the Proof of Theorem
  \ref{lem4.3}.
\end{rem}
\begin{proof}[Proof of Corollary \ref{cor4.7}]
  Let $z_0$ be an arbitrary point of $Z_0$, and let
  $R_1\Subset\Omega_1$ be a neighborhood of $z_0$.  Define this
  function $A_1:\co\to\co$:
  $$A_1(z)=\left\{\begin{array}{ll}-A(z)&z\in R_1\setminus
  Z_0\\0&z\notin R_1\setminus Z_0\end{array}\right.,$$ so $A_1\in
  L^p(\co)\cap L^2(\co)$.  The Cauchy transform $\cc(A_1)$ is in
  $\cc^{0,\alpha}(\co)$, and its distributional derivative on $\co$,
  $\partial_{\z}\cc(A_1)$, is represented by $A_1\in L^2(\co)$.  The
  function $\SSS(A_1)\in L^2(\co)$ represents the distributional
  derivative of $\cc(A_1)$ on $\co$, $\partial_z(\cc(A_1))$.  So, the
  restriction $\cc(A_1)|_{R_1}$ is bounded and in $W^{1,2}(R_1)$, with
  distributional $\z$-derivative on $R_1$ represented by $A_1|_{R_1}$.
  This is enough (\cite{gt}, \cite{bbc} \S8) for the weak chain rule
  to apply: the composite $\exp(\cc(A_1)|_{R_1})$ is in $W^{1,1}(R_1)$
  and its distributional $\z$-derivative on $R_1$ is represented by
  $\exp(\cc(A_1)|_{R_1})(A_1|_{R_1})$.

  Now Theorem \ref{lem4.3} applies to $u$ on the open set $R_1\setminus
  Z_0$, with $P=\left((u(z)-w_0)A(z)\right)|_{R_1\setminus Z_0}\in
  L^p(R_1\setminus Z_0)$.  Let $z_1$ be any point of $R_1\setminus
  Z_0$, and let $R_2\Subset R_1\setminus Z_0$ be a neighborhood of
  $z_1$.  The restriction $u|_{R_2}$ is in
  $W^{1,2}(R_2)\cap\cc^{0,\alpha}(R_2)$, and its distributional
  $\z$-derivative on $R_2$ is represented by $P|_{R_2}$, from
  (\ref{eq26}).  This is enough for the weak product rule to apply:
  the product
  $\left(\exp(\cc(A_1)|_{R_1})\right)|_{R_2}(u(z)-w_0)|_{R_2}$ is in
  $W^{1,1}(R_2)$, with distributional $\z$-derivative represented on
  $R_2$ by
  \begin{eqnarray*}
  &&\left(\exp(\cc(A_1)|_{R_1})\right)|_{R_2}P|_{R_2}+\left(\exp(\cc(A_1)|_{R_1})(A_1|_{R_1})\right)|_{R_2}(u(z)-w_0)|_{R_2}\\
  &=&\left(\exp(\cc(A_1)|_{R_1})\right)|_{R_2}\left(\left((u(z)-w_0)A(z)\right)|_{R_2}+(A_1)|_{R_2}(u(z)-w_0)|_{R_2}\right)\\
  &=&\left(\exp(\cc(A_1)|_{R_1})\right)|_{R_2}(u(z)-w_0)|_{R_2}\left((A(z))|_{R_2}+(-A(z))|_{R_2}\right)\\
  &=&0.
  \end{eqnarray*}
  By Weyl's Lemma and continuity,
  $\left(\exp(\cc(A_1)|_{R_1})\right)|_{R_2}\left(u(z)-w_0)\right)|_{R_2}$
  is holomorphic on $R_2$.  Since $z_1$ was arbitrary, every point in
  $R_1\setminus Z_0$ is contained in some neighborhood where
  $\sigma:R_1\to\co$,
  \begin{equation}\label{eq71}
    \sigma(z)=\exp(\cc(A_1)|_{R_1})((u(z)-w_0)|_{R_1}),
  \end{equation}  
  restricts to a holomorphic function, so $\sigma$ is holomorphic on
  $R_1\setminus Z_0$.  Because $\sigma$ is continuous on $R_1$ and
  equal to $0$ exactly on $R_1\cap Z_0$, Rad\'o's Theorem (\cite{n})
  implies $\sigma$ is holomorphic on $R_1$, so $z_0$ is either an
  isolated zero of $u(z)-w_0$ or $u(z)\equiv w_0$ on $R_1$.  It
  follows that the set of non-isolated points in $Z_0$ is both open
  and closed in $\Omega_1$, so it is either empty or all of
  $\Omega_1$.
\end{proof}

\section{Examples in a borderline case}\label{sec4}

The following two Examples give solutions of $\frac{\partial
  v}{\partial\z}=P(z)$ where $v$ and $P$ are continuous but $v$ is not
$\cc^1$.  This can be considered a borderline case, as $\alpha\to1^-$
in Theorem \ref{lem4.3}, or $\beta\to0^+$ in Corollary \ref{thm4.2}.
The function $v$ in Example \ref{ex4.1} is well-known and elementary,
but $\frac{\partial v}{\partial\z}$ fails to exist at one point.  The
goal of Example \ref{ex4.4} is to improve Example \ref{ex4.1} by
finding a continuous function $V$ where the partial derivatives exist
at every point, and ${\partial V}/{\partial\z}$ is continuous, while
${\partial V}/{\partial z}$ is not locally bounded.  These examples
are of interest from the point of view of the foundations of classical
complex analysis, not motivated by any particular application.
\begin{notation}\label{not3.5}
  The notation $D_{a,r}$ refers to an open disk in $\co$ with center
  $a$ and radius $r>0$.
\end{notation}
\begin{example}\label{ex4.1}
  The following function (adapted from \cite{aim} \S15.1) is
  continuous but not $\cc^1$; it satisfies the assumptions of Theorem
  \ref{lem4.3} for all $p\ge2$, and the first part of Corollary
  \ref{thm4.2}, but not the second or third.  Using the real-valued
  natural logarithm $\ln$ and positive square root, define this
  function for $z\in D_{0,1}$:
  \begin{equation}\label{eq78}
    v(z)=\left\{\begin{array}{ll}0&z=0\\z\sqrt{-\ln(|z|^2)}&0<|z|<1\end{array}\right..
  \end{equation}
  $v$ is real analytic except at the origin, where the partial
  derivatives $v_x(0)$ and $v_y(0)$ do not exist.  For $z\ne0$, the
  derivatives are:
  \begin{eqnarray}
    v_{\z}&=&\frac{-z^2}{2\sqrt{-\ln(|z|^2)}|z|^2},\label{eq75}\\
    v_z&=&\sqrt{-\ln(|z|^2)}+\frac{-1}{2\sqrt{-\ln(|z|^2)}}.\nonumber
  \end{eqnarray}
  So, $v_{\z}$ has a removable discontinuity: there is a continuous
  function $P$ equal almost everywhere to $v_{\z}$, but there is no
  continuous function equal almost everywhere to the unbounded
  function $v_z$.
\end{example}
\begin{example}\label{ex4.4}
  We start with a smooth cutoff function: let
  $\kappa:(0,\infty)\to[0,1]$ be a fixed, weakly decreasing,
  $\cc^\infty$ function satisfying $\kappa(x)\equiv1$ for
  $0<x\le\frac12$, and $\kappa(x)\equiv0$ for $x\ge
  e^{-1/2}\approx0.6$.

  Next, define the following family of functions $V_t(z):\co\to\co$,
  depending on a parameter $0<t\le\frac12$:
  $$V_t(z)=\left\{\begin{array}{ll}0&z=0\\\kappa(|z|)z|z|^{2t}\sqrt{-\ln(|z|^2)}&0<|z|<1\\0&|z|\ge1\end{array}\right..$$
  Each $V_t(z)$ is a smoothed modification of $v(z)$ from
  (\ref{eq78}): the cutoff $\kappa$ makes $V_t$ smooth on
  $\co\setminus\{0\}$, and a calculation using the positivity of the
  exponent $2t$ shows that the $x$, $y$ partial derivatives exist at
  the origin, where $\frac{\partial V_t}{\partial x}(0)=\frac{\partial
  V_t}{\partial y}(0)=0$.  A little calculus shows that $V_t(z)$ is
  bounded by a constant not depending on $t$: $|V_t(z)|\le e^{-1/2}$.

  For $0<|z|<1$, the expression:
  \begin{eqnarray}
  \frac{\partial
  V_t}{\partial\z}&=&\frac{\partial}{\partial\z}(\kappa(|z|))z|z|^{2t}\sqrt{-\ln(|z|^2)}\nonumber\\
  &&\
  +\kappa(|z|)t\frac{z}{\z}|z|^{2t}\sqrt{-\ln(|z|^2)}\label{eq93}\\
  &&\
  -\kappa(|z|)\frac12\frac{z}{\z}|z|^{2t}/\sqrt{-\ln(|z|^2)}\nonumber
  \end{eqnarray}
  shows that $\frac{\partial V_t}{\partial\z}$ is continuous on $\co$.
  The following calculation shows that $\frac{\partial
  V_t}{\partial\z}$ is bounded by a constant not depending on $t$.
  Recalling that $\kappa$ does not depend on $t$,
  $\frac{\partial}{\partial\z}(\kappa(|z|))$ is bounded by some
  $B_1>0$.
  \begin{eqnarray*}
  \left|\frac{\partial
  V_t}{\partial\z}\right|&\le&B_1\max_{\frac12\le|z|\le
  e^{-1/2}}\left\{|z|^{1+2t}\sqrt{-\ln(|z|^2)}\right\}\\ &&\
  +\max_{0<|z|\le
  e^{-1/2}}\left\{t|z|^{2t}\sqrt{-\ln(|z|^2)}\right\}\\ &&\
  +\max_{0<|z|\le
  e^{-1/2}}\left\{\frac12|z|^{2t}/\sqrt{-\ln(|z|^2)}\right\}\\
  &\le&B_1e^{-1/2}+\frac{\sqrt{t}}{\sqrt{2e}}+\frac12e^{-t}\le B_2.
  \end{eqnarray*}

  Next, choose any real sequence $R_k$ decreasing to limit $0$, and
  another positive sequence $r_k$ with $R_k-r_k>R_{k+1}+r_{k+1}$.
  (For example, $R_k=10^{-k}$ and $r_k=10^{-(k+1)}$.)

  Finally, define
  $$V(z)=\sum_{k=1}^\infty2^{-k}r_kV_{2^{-4k}}\left(\frac{z-R_ke^{\pi
  i/4}}{r_k}\right).$$ The expression
  $V_{2^{-4k}}\left(\frac{z-R_ke^{\pi i/4}}{r_k}\right)$ is the result
  of re-scaling $V_{2^{-4k}}(z)$, supported in $D_{0,1}$, to
  $V_{2^{-4k}}(z/r_k)$, supported in $D_{0,r_k}$, and then translating
  along the diagonal, so that $V_{2^{-4k}}\left(\frac{z-R_ke^{\pi
  i/4}}{r_k}\right)$ is supported in $D_{R_ke^{\pi i/4},r_k}$.  These
  support disks approach, but do not contain, the origin as
  $k\to\infty$, and are disjoint from each other, so the above
  infinite sum trivially converges for each $z\in\co$.  Every point in
  $\co$ except the origin has a neighborhood intersecting at most one
  of these disks, so $V$ is continuous and its partial
  derivatives exist on $\co\setminus\{0\}$.

  Some of the disks $D_{R_ke^{\pi i/4},r_k}$ may intersect the $x$ and
  $y$ axes, but by the choice of the cutoff function $\kappa$ and the
  numerical inequality $e^{-1/2}<1/\sqrt{2}$, the support of
  $V_{2^{-4k}}\left(\frac{z-R_ke^{\pi i/4}}{r_k}\right)$ is actually
  contained in $D_{R_ke^{\pi i/4},r_k/\sqrt{2}}$, which is contained
  in the open first quadrant and disjoint from the $x$-axis and the
  $y$-axis.  The partial derivatives of $V$ exist at the origin, where
  $\frac{\partial V}{\partial x}(0)=\frac{\partial V}{\partial
  y}(0)=0$, because $V\equiv0$ along both the axes.
 
  $V$ is continuous at the origin,
  $\displaystyle{\lim_{z\to0}V(z)=V(0)=0}$, and in fact satisfies the
  stronger condition of complex differentiability at that point:
  $\displaystyle{\lim_{z\to0}\frac{V(z)}{z}=0}$.  For
  $z$ in the $k^{th}$ disk, $V(z)=0$ for $\left|z-R_ke^{\pi
  i/4}\right|>e^{-1/2}r_k$, and otherwise,
  \begin{eqnarray}
    \frac{\left|V(z)\right|}{|z|}&=&\frac{\left|2^{-k}r_kV_{2^{-4k}}\left(\frac{z-R_ke^{\pi
          i/4}}{r_k}\right)\right|}{|z|}\label{eq106}\\
          &\le&\frac{\displaystyle{2^{-k}r_k\max_{z\in\co}\left|V_{2^{-4k}}(z)\right|}}{R_k-e^{-1/2}r_k}\nonumber\\
          &\le&\frac{2^{-k}e^{-1/2}}{\frac{R_k}{r_k}-e^{-1/2}}\le\frac{2^{-k}e^{-1/2}}{1-e^{-1/2}}.\nonumber
  \end{eqnarray}

  The derivative $\frac{\partial V}{\partial\z}$ is continuous at
  every point of $\co$, including the origin.  To show
  $$\lim_{z\to0}\frac{\partial V}{\partial\z}(z)=\frac{\partial
  V}{\partial\z}(0)=0,$$ apply the chain rule at an arbitrary point
  $z_0$ in any particular disk:
  \begin{eqnarray*} \frac{\partial
  V}{\partial\z}(z_0)&=&\left.\frac{\partial}{\partial\z}\left(2^{-k}r_kV_{2^{-4k}}\left(\frac{z-R_ke^{\pi
  i/4}}{r_k}\right)\right)\right|_{z=z_0}\\ &=&2^{-k}r_k\frac{\partial
  V_{2^{-4k}}}{\partial\z}\left(\frac{z_0-R_ke^{\pi
  i/4}}{r_k}\right)\frac1{r_k}\\ \implies \left|\frac{\partial
  V}{\partial\z}(z_0)\right|&\le&2^{-k}B_2.
  \end{eqnarray*}

  Now, consider the $z$-derivative.  For $0<|z|<1$, the expression:
  \begin{eqnarray}
  \frac{\partial V_t}{\partial
  z}&=&\frac{\partial}{\partial\z}(\kappa(|z|))z|z|^{2t}\sqrt{-\ln(|z|^2)}\nonumber\\
  &&\ +\kappa(|z|)(1+t)|z|^{2t}\sqrt{-\ln(|z|^2)}\label{eq92}\\ &&\
  -\kappa(|z|)\frac12|z|^{2t}/\sqrt{-\ln(|z|^2)}\nonumber
  \end{eqnarray}
  shows that $\frac{\partial V_t}{\partial z}$ is continuous on $\co$.
  However, the coefficient $(1+t)$ in term (\ref{eq92}) is
  significantly larger than the corresponding coefficient $t$ in
  (\ref{eq93}) as $t\to0^+$; this gain is the key step in this
  example.  The continuity of both $\frac{\partial V_t}{\partial z}$
  and $\frac{\partial V_t}{\partial\z}$ imply that $V_t$ is $\cc^1$ on
  $\co$, and $V$ is $\cc^1$ on $\co\setminus\{0\}$.  This, together
  with (\ref{eq106}), shows that $V$ is differentiable (real
  differentiable, in the sense of multivariable calculus) on $\co$.

  To show $\frac{\partial V}{\partial z}$ is not locally bounded,
  define a sequence
  $$z_k=r_ke^{\left(-2^{4k-2}\right)}+R_ke^{\pi i/4},$$ so
  $\displaystyle{\lim_{k\to\infty}z_k=0}$.
  \begin{eqnarray*}
    &&\frac{\partial V}{\partial z}(z_k)\\
    &=&\left.\frac{\partial}{\partial
    z}\left(2^{-k}r_kV_{2^{-4k}}\left(\frac{z-R_ke^{\pi
    i/4}}{r_k}\right)\right)\right|_{z=z_k}\\
    &=&2^{-k}r_k\frac{\partial V_{2^{-4k}}}{\partial
    z}\left(e^{\left(-2^{4k-2}\right)}\right)\frac1{r_k}\\
    &=&2^{-k}\left(0+\left(1+2^{-4k}\right)\left|e^{\left(-2^{4k-2}\right)}\right|^{\left(2^{-4k+1}\right)}\sqrt{-\ln\left(\left|e^{\left(-2^{4k-2}\right)}\right|^2\right)}\right.\\
    &&\
    \left.-\frac12\left|e^{\left(-2^{4k-2}\right)}\right|^{\left(2^{-4k+1}\right)}/\sqrt{-\ln\left(\left|e^{\left(-2^{4k-2}\right)}\right|^2\right)}\right)\\
    &=&2^{-k}\left(\left(1+2^{-4k}\right)e^{-1/2}\sqrt{2^{4k-1}}-\frac12e^{-1/2}/\sqrt{2^{4k-1}}\right)\\
    &=&2^k/\sqrt{2e}.
  \end{eqnarray*}
  It follows that the derivatives $\frac{\partial V}{\partial x}$,
  $\frac{\partial V}{\partial y}$ are also not locally bounded.
\end{example}
\begin{rem}\label{rem3.6}
  The above piecewise construction, with smooth cutoffs and a sequence
  of exponents $t=2^{-4k}$, is similar to examples constructed in
  \cite{rosay} and \cite{cp}, of $\cc^\infty$ vector valued functions
  where the $\z$-derivative is small compared to the $z$-derivative.
\end{rem}
\begin{rem}\label{rem3.4}
  It is well-known that the Beurling transform $\SSS:v_{\z}\mapsto
  v_z$ need not preserve the $\cc^0$ or $L^\infty$ properties.
  Example \ref{ex4.4} shows that this still holds even when $v_{\z}$
  is the continuous derivative of a differentiable function.
\end{rem}
\begin{rem}\label{rem3.5}
  In the theory of one real variable, the function $x^2\sin(1/x^2)$
  extends to a differentiable function with an unbounded derivative.
  We do not know of an analogous {\underline{elementary}} expression
  in $x$ and $y$ with the same properties as $V(z)$.  Any function
  where $v_{\z}=\frac12(v_x+iv_y)$ is locally bounded, and
  $v_z=\frac12(v_x-iv_y)$ is not, cannot be real valued; $v$ must be
  complex valued.
\end{rem}

\section{The separable Cauchy-Riemann equation}\label{sec1}

Let $\Omega_1$ and $\Omega_2$ be open subsets of $\co$.  Here we
consider the ``separable'' case of the nonlinear Cauchy-Riemann
equation where the RHS of (\ref{eq29}) factors in the form:
\begin{equation}\label{eq30}
 \frac{\partial u}{\partial\z}=f(u)g(z),
\end{equation}
for $u:\Omega_1\to\Omega_2$, $g:\Omega_1\to\co$, and where
$f:\Omega_2\to\co$ is holomorphic.  We have already considered one
separable equation in Corollary \ref{cor4.7}.  The goal is to compute
explicit (as in (\ref{eq96}), (\ref{eq100})), or implicit (as in
(\ref{eq70})) local formulas for all solutions $u$ of (\ref{eq30})
satisfying minimal regularity properties.

We consider three cases: first, where $f$ is nonvanishing, Subsection
\ref{subsec2} uses only results of single-variable complex analysis
(as in \cite{c} Ch.\ IV and \cite{n}) without appealing to integral
transforms as in Section \ref{sec2}.  Second, where $f$ has a simple
zero, Subsection \ref{subsec4.2} solves an auxiliary ODE (\ref{eq97})
to find a substitution that establishes existence and uniqueness for
(\ref{eq30}).  It is not until the third case, where $f$ has a zero of
multiplicity greater than one, that we need to use the results of
Section \ref{sec2}, in Subsection \ref{subsec3}.

\subsection{Nonvanishing $f$}\label{subsec2}

\ 

The following Lemma is an existence result; it is essentially the
first-year calculus method for solving a separable first-order ODE.
\begin{lem}\label{thm2.2}
  For functions $f:\Omega_2\to\co$ and $g:\Omega_1\to\co$, suppose
  there exist a holomorphic function $F:\Omega_2\to\co$ such that
  $\frac{\partial F}{\partial w}=\frac1{f(w)}$, and a continuous
  function $G:\Omega_1\to\co$ such that the partial derivatives $G_x$,
  $G_y$ exist and satisfy $\frac{\partial G}{\partial\z}=g(z)$.  For
  any points $z_0\in\Omega_1$, $w_0\in\Omega_2$, there exists
  a non-constant function $u:\Omega_1^0\to\Omega_2$ on some
  neighborhood of $z_0$, $\Omega_1^0\subseteq\Omega_1$, such that
  $\frac{\partial u}{\partial\z}=f(u)g(z)$ and $u(z_0)=w_0$.
\end{lem}
\begin{proof}
  For the existence of a primitive $F$, it is necessary that $f$ is
  holomorphic and nonvanishing on $\Omega_2$, and it would further be
  sufficient for $\Omega_2$ to be simply connected.

  Because $F^\prime(w_0)=\frac1{f(w_0)}\ne0$, there is some
  neighborhood $\Omega_2^0$ of $w_0$ such that $F$ is one-to-one on
  $\Omega_2^0$, the image $F(\Omega_2^0)$ is an open subset of $\co$,
  and there is a holomorphic local inverse
  $H:F(\Omega_2^0)\to\Omega_2^0$.  The derivative of $H$ is
  $H^\prime(\zeta)=\frac1{F^\prime(H(\zeta))}=f(H(\zeta))$.

  Let $\Omega_1^1$ be any neighborhood of $z_0$ in $\Omega_1$, and let
  $\theta:\Omega_1^1\to\co$ be any holomorphic function.  The following
  function, 
  \begin{equation}\label{eq4}
    G_1(z)=G(z)-G(z_0)+\theta(z)-\theta(z_0)+F(w_0),
  \end{equation}
  is continuous on $\Omega_1^1$ and satisfies $G_1(z_0)=F(w_0)$.
  The set
  $$\Omega_1^0=G_1^{-1}(F(\Omega_2^0))=\{z\in\Omega_1^1:G_1(x)\in
  F(\Omega_2^0)\}$$ is an open neighborhood of $z_0$, and is the
  domain of the composite function 
  \begin{equation}\label{eq95}
    u=H\circ(G_1|_{\Omega_1^0}).
   \end{equation}
  By construction, $u(z_0)=w_0$, and $$\frac{\partial
  u}{\partial\z}=H^\prime(G_1(z))\frac{\partial
  G_1}{\partial\z}=f(H(G_1(z)))g(z)=f(u)g(z).$$ Similarly,
  $$\frac{\partial u}{\partial z}=H^\prime(G_1(z))\frac{\partial
    G_1}{\partial z}=f(u)\left(\frac{\partial G}{\partial
    z}+\frac{\partial \theta}{\partial z}\right),$$ and $\theta$ can
  be chosen so that the derivative is non-zero at $z_0$.
\end{proof}
The above method constructs a local solution, of the form $u=H\circ
G_1$, which has the same $\cc^r$ regularity as $G$.  Theorem
\ref{thm2.4} shows all continuous solutions are locally of the same
form, using the following Lemma and the $(*)$ property.
\begin{lem}\label{lem2.3}
  For $f$ and $F$ as in Lemma \ref{thm2.2} and any
  $g:\Omega_1\to\co$, suppose $u:\Omega_1\to\Omega_2$ and
  $v:\Omega_1\to\Omega_2$ are continuous functions both satisfying
  property {\rm{$(*)$}}, and
  $$\frac{\partial u}{\partial\z}=f(u)g(z), \ \ \frac{\partial
  v}{\partial\z}=f(v)g(z)$$ almost everywhere in $\Omega_1$.  Then
  there exists $C:\Omega_1\to\co$ which is holomorphic and satisfies
  $F(v(z))=F(u(z))+C(z)$.
\end{lem}
\begin{proof}  
  Applying the chain rule off the union (still countable) of the
  exceptional sets for $u$ and $v$ from $(*)$, and Proposition
  \ref{prop2.4},
  \begin{eqnarray*}
    \frac{\partial}{\partial\z}(F(v)-F(u))&=&F^\prime(v(z))\frac{\partial
      v}{\partial\z}-F^\prime(u(z))\frac{\partial u}{\partial\z}\\
      &=_{a.e.}&\frac1{f(v(z))}f(v(z))g(z)-\frac1{f(u(z))}f(u(z))g(z)\\
    &\equiv&0\\
      \implies F(v)-F(u)&=&C(z).
  \end{eqnarray*}
\end{proof}
The following Theorem is stated as a regularity result, but our main
interest is in uniqueness --- showing that, for nonvanishing $f$, all
continuous solutions of (\ref{eq30}) that satisfy $(*)$ must be
locally of the form (\ref{eq96}).
\begin{thm}\label{thm2.4}
  Let $f:\Omega_2\to\co$ be holomorphic, and suppose
  $g:\Omega_1\to\co$ is equal to $\frac{\partial G}{\partial\z}$ for
  some $G\in\cc^r(\Omega_1)$, $r=0,1,2,\ldots,\infty,\omega$.  For a
  continuous function $v:\Omega_1\to\Omega_2$, define the open set
  $\Omega_0=\{z\in\Omega_1:f(v(z))\ne0\}$.  If $v$ satisfies
  {\rm{$(*)$}} on $\Omega_0$ and $\frac{\partial
  v}{\partial\z}=f(v)g(z)$ almost everywhere in $\Omega_0$, then
  $v\in\cc^r(\Omega_0)$.
\end{thm}
\begin{proof}
  Let $z_0$ be an arbitrary point in $\Omega_0$, so $f(v(z_0))\ne0$
  and there is some simply connected neighborhood of $v(z_0)$,
  $\Omega_2^1\subseteq\Omega_2$, so that $f$ is nonvanishing on
  $\Omega_2^1$.  There exists a holomorphic $F:\Omega_2^1\to\co$ such
  that $\frac{\partial F}{\partial w}=\frac1{f(w)}$.  Let
  $\Omega_0^1=v^{-1}(\Omega_2^1)$, so $\Omega_0^1$ is an open
  neighborhood of $z_0$ in $\Omega_0$.  Lemma \ref{thm2.2} applies to
  the restrictions $g:\Omega_0^1\to\co$ and $f:\Omega_2^1\to\co$.
  There exists a solution $u:\Omega_0^2\to\Omega_2^1$ on some
  neighborhood of $z_0$, $\Omega_0^2\subseteq\Omega_0^1$, such that
  $\frac{\partial u}{\partial\z}=f(u)g(z)$, $u(z_0)=v(z_0)$, and
  $F\circ u=G_1|_{\Omega_0^2}$, where $G_1(z)=G(z)-G(z_0)+F(v(z_0))$
  (from (\ref{eq4}) with $\theta\equiv0$; for this Theorem, $u$ is not
  necessarily non-constant.)  From (\ref{eq95}), where $u$ is defined
  as a composite of a holomorphic function with $G_1$,
  $u\in\cc^r(\Omega_0^2)$.  By Lemma \ref{lem2.3}, there exists a
  holomorphic function $C:\Omega_0^2\to\co$ such that
  $F(v(z))=F(u(z))+C(z)$ and $C(z_0)=0$.  As in the Proof of Lemma
  \ref{thm2.2}, there is some neighborhood of $v(z_0)$,
  $\Omega_2^2\subseteq\Omega_2^1$, where $F$ is one-to-one, so
  $F(\Omega_2^2)$ is open in $\co$ and $H:F(\Omega_2^2)\to\Omega_2^2$
  is a holomorphic local inverse of $F$.  Define this open
  neighborhood of $z_0$,
  $$\Omega_0^3=(F\circ
  u+C)^{-1}(F(\Omega_2^2))=\{z\in\Omega_0^2:F(u(z))+C(z)\in
  F(\Omega_2^2)\}.$$ Then, for all $z\in\Omega_0^3$,
  $F(v(z))=F(u(z))+C(z)\in F(\Omega_2^2)$, and plugging into $H$ gives
  \begin{eqnarray}
    v(z)&=&H(F(u(z))+C(z))\label{eq96}\\ &=&H(G_1(z)+C(z))\nonumber.
  \end{eqnarray}
  It follows from (\ref{eq96}) that $v\in\cc^r(\Omega_0^3)$, which
  since $z_0$ was arbitrary, is enough to show $v\in\cc^r(\Omega_0)$.
\end{proof}
\begin{example}\label{ex3.4}
  Let $f(w)=e^w$, $g(z)\equiv1$, and choose $F(w)=-e^{-w}$, $G(z)=\z$.
  Let $\Omega_2^0$ be a neighborhood of $w_0\in\co$ where $F$ is
  one-to-one, so there is a branch of the complex logarithm which is a
  holomorphic local inverse of $F$, $H(\zeta)=-{\rm Log}(-\zeta)$.
  Then, for any $z_0\in\co$, if $\Omega_1$ is a neighborhood of $z_0$,
  and $v:\Omega_1\to\co$ is continuous, satisfies $(*)$, is a solution
  of $$\frac{\partial v}{\partial\z}=e^v$$ almost everywhere in
  $\Omega_1$, with initial condition $v(z_0)=w_0$, then by Theorem
  \ref{thm2.4}, $v$ is real analytic on $\Omega_1$, and locally near
  $z_0$, $$v(z)=-{\rm
    Log}\left(-(\overline{(z-z_0)}+C(z)-e^{-w_0})\right),$$ for some
  holomorphic function $C(z)$ with $C(z_0)=0$.  Conversely, choosing
  any such $C$ gives an example of a local solution.  One such
  solution, with $z_0=w_0=0$ and $C(z)=z$, is real valued on the
  domain $\{\rp(z)<\frac12\}$, $$v(x+iy)=-\ln(-2x+1).$$ The level
  sets are lines, unlike the isolated points as in Corollary
  \ref{cor4.7}.
\end{example}

\subsection{Simple zeros of $f$}\label{subsec4.2}

\ 

Informally considering the equation $\frac{\partial
  u}{\partial\z}=ug(z)$, the obvious solutions are of the form
$u(z)=B(z)\exp(G(z))$, where $B$ is holomorphic and $G$ is a
$\z$-antiderivative of $g$.  To apply this idea to the more general
separable equation $\frac{\partial u}{\partial\z}=f(u)g(z)$, where $f$
has a simple zero, the following Lemma leads to a useful substitution.

\begin{lem}\label{lem4.5}
  Given $f:\Omega_2\to\co$ holomorphic, with a simple zero at $w_0$,
  there exist a disk $D_{0,r_0}$ and a holomorphic function
  $h:D_{0,r_0}\to\Omega_2$ such that $h(0)=w_0$, $h:D_{0,r_0}\to
  h(D_{0,r_0})$ is invertible, and for $\zeta\in D_{0,r_0}$,
  \begin{equation}\label{eq99}
    f^\prime(w_0)\zeta h^\prime(\zeta)=f(h(\zeta)).
  \end{equation}
\end{lem}
\begin{proof}
  On some disk $D_{w_0,r_1}\subseteq\Omega_2$,
  $\displaystyle{f(w)=\sum_{j=1}^\infty f_j(w-w_0)^j}$, where by
  hypothesis, $f_1=f^\prime(w_0)\ne0$.  On $D_{w_0,r_1}$, define
  $$\tilde
  f(w)=\sum_{j=2}^\infty\left(\frac{f_j}{f_1}\right)(w-w_0)^j,$$ so
  $\tilde f$ is holomorphic and $f(w)=f_1(w-w_0+\tilde f(w))$.  For
  some $r_2>0$, $r_3>0$ with $r_2(1+r_3)<r_1$, the following
  two-variable function is holomorphic and bounded on the bidisk
  $D_{0,r_2}\times D_{0,r_3}\subseteq\co^2$, defined by an absolutely
  convergent power series:
  \begin{eqnarray}
    {\mathbf F}(W_1,W_2)&=&\left\{\begin{array}{cl}\frac1{W_1^2}\tilde
  f(w_0+W_1+W_1W_2)&W_1\ne0\\
  \frac{f_2}{f_1}(1+W_2)^2&W_1=0\end{array}\right.\nonumber\\
  &=&\sum_{j=2}^\infty\left(\frac{f_j}{f_1}\right)W_1^{j-2}(1+W_2)^j=\sum_{j,\ell}F_{j\ell}W_1^jW_2^\ell.\label{eq101}
  \end{eqnarray}
  The differential equation 
  \begin{equation}\label{eq97}
   H^\prime(\zeta)={\mathbf F}(\zeta,H(\zeta)),
  \end{equation}
   with initial condition $H(0)=0$, has a formal solution
   $\displaystyle{H(\zeta)=\sum_{j=1}^\infty H_j\zeta^j}$, where the
   coefficient sequence $H_j$ is defined uniquely by the coefficients
   $F_{j\ell}$, which uniquely depend on $f_1,f_2,\ldots$ by
   re-centering the power series in step (\ref{eq101}).  The series is
   convergent (\cite{hille2} Theorem 2.5.1, proved by a majorization
   method), so $H$ is holomorphic on some disk $D_{0,r_4}$.  Define
   $$h(\zeta)=w_0+\zeta+\zeta H(\zeta),$$ so $h$ is holomorphic and
   invertible on some disk $D_{0,r_0}$, and by construction,
   satisfies:
  \begin{eqnarray*}
    f^\prime(w_0)\zeta h^\prime(\zeta)&=&f_1\zeta(1+H(\zeta)+\zeta H^\prime(\zeta))\\
    &=&f_1(\zeta+\zeta H(\zeta)+\zeta^2{\mathbf F}(\zeta,H(\zeta)))\\
    &=&f_1(h(\zeta)-w_0+\tilde f(w_0+\zeta+\zeta H(\zeta)))\\
    &=&f(h(\zeta)).
  \end{eqnarray*}
\end{proof}
\begin{rem}\label{rem4.6}
  Equation (\ref{eq99}) is a special case of an equation considered by
  \cite{hille2} \S11.1; the above proof shows how $h$ can be computed
  in terms of $f$.
\end{rem}
\begin{thm}\label{lem4.6}
  Given $f:\Omega_2\to\co$ holomorphic with only simple zeros,
  $g:\Omega_1\to\co$, and any points $z_0\in\Omega_1$,
  $w_0\in\Omega_2$, if there is a continuous function
  $G:\Omega_1\to\co$ such that $\frac{\partial G}{\partial\z}=g(z)$,
  then there exists a non-constant, continuous function
  $u:\Omega_1^0\to\Omega_2$ on some neighborhood of $z_0$,
  $\Omega_1^0\subseteq\Omega_1$, such that $\frac{\partial
    u}{\partial\z}=f(u)g(z)$ and $u(z_0)=w_0$.
\end{thm}
\begin{proof}
  For the case where $f$ is non-vanishing at $w_0$, Lemma \ref{thm2.2}
  applies locally near $w_0$, so we assume that $f$ has a zero of
  order $1$ at $w_0$.

  Let $f_1=f^\prime(w_0)$ and $h:D_{0,r_0}\to\Omega_2$ be as in Lemma
  \ref{lem4.5}.  The function $U(z)=(z-z_0)\exp(f_1G(z))$ is
  continuous on $\Omega_1$, with $U(z_0)=0$.  Let
  $\Omega_1^0=U^{-1}(D_{0,r_0})$, and define
  $u=h\circ(U|_{\Omega_1^0})$, so that
  \begin{equation}\label{eq98}
    u(z)=h(U(z))=h((z-z_0)\exp(f_1G(z)))
  \end{equation}
  is continuous on $\Omega_1^0$, and its partial derivatives satisfy,
  using (\ref{eq99}):
  \begin{eqnarray*}
    \frac{\partial
    u}{\partial\z}&=&h^\prime((z-z_0)\exp(f_1G(z)))(z-z_0)\exp(f_1G(z))f_1\frac{\partial
    G}{\partial\z}\\ &=&f(h((z-z_0)\exp(f_1G(z))))g(z)=f(u)g(z),\\
    \left.\frac{\partial u}{\partial
    z}\right|_{z=z_0}&=&h^\prime(0)\exp(f_1G(z_0))\ne0.
  \end{eqnarray*}
\end{proof}
As remarked after Lemma \ref{thm2.2}, for any $f$, the solution $u$
constructed in (\ref{eq98}) has the same $\cc^r$ regularity as the
antiderivative $G$.

The following Theorem is a generalization of Theorem \ref{thm2.4};
again our interest is in computing a local formula (\ref{eq100}) for
any continuous solution.
\begin{thm}\label{thm4.7}
  Let $f:\Omega_2\to\co$ be holomorphic, with only simple zeros, and
  suppose $g:\Omega_1\to\co$ is equal to $\frac{\partial
  G}{\partial\z}$ for some $G\in\cc^r(\Omega_1)$,
  $r=0,1,2,\ldots,\infty,\omega$.  If $v:\Omega_1\to\Omega_2$ is
  continuous, satisfies {\rm{$(*)$}}, and $\frac{\partial
  v}{\partial\z}=f(v)g(z)$ almost everywhere in $\Omega_1$, then
  $v\in\cc^r(\Omega_1)$.
\end{thm}
\begin{proof}
  Let $z_0$ be an arbitrary point in $\Omega_0$.  If $f(v(z_0))\ne0$,
  then Theorem \ref{thm2.4} applies, to show that there is some
  neighborhood of $z_0$ where $v$ is $\cc^r$.  Otherwise, $w_0=v(z_0)$
  is a simple zero of $f$, so Lemma \ref{lem4.5} applies to give $h$
  on $D_{0,r_0}$.  Let $\Omega_1^1=v^{-1}(h(D_{0,r_0}))$ be a
  neighborhood of $z_0$ in $\Omega_1$.  The function
  $B:\Omega_1^1\to\co$ defined by 
  \begin{equation}\nonumber
    B(z)=h^{-1}(v(z))\exp(-f_1G(z))
  \end{equation}
  is continuous, satisfies $(*)$ and $B(z_0)=0$, and almost everywhere
  in $\Omega_1^1$,
  \begin{eqnarray}
    &&\frac{\partial
      B}{\partial\z}\nonumber\\ &=_{a.e.}&\frac{\partial
      v/\partial\z\exp(-f_1G(z))}{h^\prime(h^{-1}(v(z)))}+h^{-1}(v(z))\exp(-f_1G(z))(-f_1)\frac{\partial
      G}{\partial\z}\nonumber\\ &=_{a.e.}&\frac{f(v)g(z)\exp(-f_1G(z))}{h^\prime(h^{-1}(v(z)))}-B(z)f_1g(z)\nonumber\\ &=&f_1h^{-1}(v(z))g(z)\exp(-f_1G(z))-B(z)f_1g(z)\equiv0.\nonumber
  \end{eqnarray}
  The conclusion is that, on $\Omega_1^1$, $B$ is holomorphic by
  Proposition \ref{prop2.4}, and
  \begin{equation}\label{eq100}
    v(z)=h(B(z)\exp(f_1G(z)))
  \end{equation}
  is $\cc^r$, so $v$ is $\cc^r$ on a neighborhood of every point in
  $\Omega_1$.
\end{proof}
Any holomorphic $B$ with $B(z_0)=0$ in (\ref{eq100}) gives a local
solution $v$, by the same calculation as the example in Theorem
\ref{lem4.6}.

\subsection{Zeros of $f$ with higher multiplicity}\label{subsec3}

\

The result of the following Theorem is a local implicit formula
(\ref{eq70}) for continuous solutions $u$ of the separable equation
$u_{\z}=f(u)g(z)$, when $f$ has a zero of order $>1$.  Unlike the
expression from Equation (\ref{eq100}), involving a substitution
function $h$ depending on $f$, the expression in Equation (\ref{eq70})
uses only antiderivatives $F$ and $G$ of the given factors $f$ and
$g$.
\begin{thm}\label{thm2.9}
  Given an open set $\Omega_1\subseteq\co$, $p>2$, and $g\in
  L^p_{loc}(\Omega_1)$, suppose there is some $G:\Omega_1\to\co$ so
  that $G$ is continuous, satisfies {\rm{$(*)$}}, and $\frac{\partial
    G}{\partial\z}=g(z)$ almost everywhere.  Let $f:\Omega_2\to\co$ be
  continuous with $f(w_0)=0$, and suppose for some disk
  $D_{w_0,r}\subseteq\Omega_2$, there is
  $F:D_{w_0,r}\setminus\{w_0\}\to\co$ so that $F$ is holomorphic and
  satisfies $F^\prime(w)=\frac1{f(w)}$.  For any $z_0\in\Omega_1$, if
  there exist a neighborhood of $z_0$, $\Omega_1^0\subseteq\Omega_1$,
  and a non-constant, continuous function $u:\Omega_1^0\to\Omega_2$
  satisfying {\rm{$(*)$}}, $u(z_0)=w_0$, and $\frac{\partial
    u}{\partial\z}=f(u)g(z)$ almost everywhere on $\Omega_1^0$, then
  there exist an integer $M\ge1$ and a nonvanishing holomorphic
  function $\phi(z)$ on some neighborhood of $z_0$, with
  \begin{equation}\label{eq70}
    (z-z_0)^{M}F(u(z))=\phi(z)+(z-z_0)^{M}G(z).
  \end{equation}
\end{thm}
\begin{proof}
  From $F^\prime=\frac1f$, $f$ is holomorphic and nonvanishing on
  $D_{w_0,r}\setminus\{w_0\}$, and because $f$ is continuous, $f$ is
  holomorphic on $D_{w_0,r}$.  Let $k\ge1$ be the order of vanishing
  of $f(w)$ at $w_0$, so there is a series expression converging on
  $D_{w_0,r}$,
  \begin{equation}\label{eq9}
    f(w)=(w-w_0)^k(f_k+\sum_{j=k+1}^\infty f_{j}(w-w_0)^{j-k})
  \end{equation}
  with $f_k\ne0$.  The reciprocal has a Laurent expansion
  \begin{equation}\label{eq10}
    \frac{1}{f(w)}=(w-w_0)^{-k}(\frac1{f_k}+\sum_{\ell=1}^\infty
    q_{\ell}(w-w_0)^\ell).
  \end{equation}
  The existence of the primitive $F$ is equivalent to $k>1$ and
  $q_{k-1}=0$ (this is the Residue of $\frac1f$ at $w_0$).  By
  integrating the above Laurent series, any holomorphic primitive $F$
  has a pole of order exactly $k-1$ at $w_0$.  So,
  \begin{equation}\label{eq91}
    (w-w_0)^{k-1}F(w)
  \end{equation}
   extends to a holomorphic function on $D_{w_0,r}$, which is
  nonvanishing on some possibly smaller disk $D_{w_0,r_0}$; denote the
  extension $\tilde F:D_{w_0,r_0}\to\co\setminus\{0\}$.

  Let $R_0\Subset u^{-1}(D_{w_0,r_0})\subseteq\Omega_1^0$ be a
  neighborhood of $z_0$, so that, using (\ref{eq9}), $u$ satisfies the
  hypotheses of Corollary \ref{cor4.7} on $R_0$.  As in (\ref{eq71}),
  on a neighborhood of $z_0$, $R_1\Subset R_0$,
  $u(z)-w_0=e^{U}\sigma$, where $\sigma$ is holomorphic on $R_1$ and
  $U\in\cc^{0,\alpha}(R_1)$.  On $R_1$, the composite $\tilde F(u(z))$
  is continuous and nonvanishing.  There is a neighborhood of $z_0$,
  $\Omega_1^1\subseteq R_1$, where $z_0$ is the only point where
  $u(z)=w_0$, the holomorphic factor $\sigma$ has a series expansion
  at $z_0$ with order of vanishing $m\ge1$, and
  \begin{equation}\label{eq7}
    u(z)-w_0=(z-z_0)^mp(z),
  \end{equation}
  for some nonvanishing continuous function $p(z)$.  An expression for
 $\tilde F(u(z))$ can be computed on $\Omega_1^1\setminus\{z_0\}$,
 using (\ref{eq91}) and (\ref{eq7}):
  $$\tilde
  F(u(z))=(u(z)-w_0)^{k-1}F(u(z))=((z-z_0)^mp(z))^{k-1}F(u(z)).$$ It
  follows that the product $$(z-z_0)^{m(k-1)}F(u(z))$$ extends from
  $\Omega_1^1\setminus\{z_0\}$ to a nonvanishing, continuous function
  on $\Omega_1^1$, and by the continuity of $G$, the expression
  \begin{equation*}
    \phi=(z-z_0)^{m(k-1)}F(u(z))-(z-z_0)^{m(k-1)}G(z)
  \end{equation*}
   is also nonvanishing and continuous on some neighborhood of $z_0$,
   $\Omega_1^2\subseteq\Omega_1^1$.  For all $z$ except $z_0$ and
   possibly countably many more from the exceptional sets from $(*)$
   for $u$ and $G$,
  \begin{eqnarray*}
    \frac{\partial}{\partial\z}\phi(z)&=&\frac{\partial}{\partial\z}\left((z-z_0)^{m(k-1)}(F(u(z))-G(z))\right)\\
    &=&(z-z_0)^{m(k-1)}\left(F^\prime(u(z))\frac{\partial
    u}{\partial\z}-\frac{\partial G}{\partial\z}\right)\\
    &=_{a.e.}&(z-z_0)^{m(k-1)}\left(\frac1{f(u(z))}f(u(z))g(z)-g(z)\right)\equiv0,
  \end{eqnarray*}
  so by Proposition \ref{prop2.4}, $\phi$ is holomorphic on
  $\Omega_1^2$.
\end{proof}
Note that the exponent $M=m(k-1)$ depends on the order of vanishing of
$u$ and $f$, but not on the choices of primitives $F$ and $G$.
\begin{example}\label{ex2.10}
  Let $0<\alpha<1$, $g=|z|^{-1+\alpha}\in L^p_{loc}(\co)$, $f(w)=w^2$,
  and $w_0=0$, so $k=2$ and $p>2$.  Choose antiderivatives
  $F(w)=-w^{-1}$, and $G(z)=\frac2{1+\alpha}\z|z|^{-1+\alpha}$
  extended to $G(0)=0$.  Then for any $z_0$, if $u$ is a non-constant,
  continuous solution of
  \begin{equation}\label{eq89}
    \frac{\partial u}{\partial\z}=u^2|z|^{-1+\alpha}
  \end{equation}
   almost everywhere, satisfying $(*)$ and $u(z_0)=0$, then there
   exist a positive integer $m$ and a holomorphic function $\phi$ with
   $\phi(z_0)\ne0$, so that for non-zero $z$ near $z_0$,
  \begin{eqnarray}
    (z-z_0)^m(-u(z))^{-1}&=&\phi(z)+(z-z_0)^mG(z)\label{eq90}\\ \implies
    u(z)&=&\displaystyle{\frac{-(z-z_0)^m}{\phi(z)+(z-z_0)^m\frac2{1+\alpha}\z|z|^{-1+\alpha}}}.\label{eq88}
  \end{eqnarray}
  In this case, choosing any $m$ and $\phi$ gives an example of a
  local solution $u$ with order of vanishing $m$ as in (\ref{eq7}).
  When extended by continuity to $u(0)=\frac{-(-z_0)^m}{\phi(0)}$, $u$
  is H\"older continuous on rectangles, as in Theorem \ref{lem4.3},
  and if $g$ and $u$ are restricted to a domain not containing $z=0$,
  then $g$ and $u$ are real analytic, as in Corollary \ref{thm4.2}.
\end{example}
\begin{example}\label{ex3.7}
  If, in Example \ref{ex2.10}, $\alpha=1$, then (\ref{eq89}) becomes
  the autonomous equation $\displaystyle{\frac{\partial
  u}{\partial\z}=u^2}$.  All solutions with initial condition
  $u(z_0)=0$ are real analytic, but the form of the solution set does
  not change: non-constant solutions still satisfy $(\ref{eq88})$,
  with $\alpha=1$.  Equations with higher powers,
  $\displaystyle{\frac{\partial u}{\partial\z}=u^k}$, have implicit
  solutions (\ref{eq70}) similar to (\ref{eq90}), but require
  selecting a local root to get an explicit solution for $u$ as in
  (\ref{eq88}).
\end{example}

\section{An application to almost complex geometry}\label{sec3}

\subsection{Normal coordinates in $\re^4$}

\ 

Let $J(\vec x)$ be a smooth almost complex structure on a neighborhood
of the origin in $\re^4$.  For example, if $J_{std}$ is the $2\times2$
constant matrix $\left[\begin{array}{cc}0&-1\\1&0\end{array}\right]$,
then the constant matrix
$J_0=\left[\begin{array}{cc}J_{std}&0\\0&J_{std}\end{array}\right]_{4\times4}$
is the standard complex structure operator for $\co^2=(\re^4,J_0)$.

For an open set $\Omega\subseteq\co=(\re^2,J_{std})$, a
$J$-holomorphic curve is a differentiable map ${\bf
u}:\Omega\to\re^4$, so that the differential $d{\bf u}$ satisfies
$d{\bf u}(x,y)\circ J_{std}=J({\bf u}(x,y))\circ d{\bf u}(x,y)$.

We very briefly recall the geometric construction of ``normal
coordinates'' from \cite{sikorav}, \cite{st1}, \cite{taubes}, but
then, starting with Equation (\ref{eq60}), go into some detail
regarding computations in this coordinate system.

Near a given point $Z_0$ on an embedded $J$-holomorphic curve ${\bf
  u}$, there exists a family of local perturbations of the curve,
parametrized by a complex variable $w$, which together with a complex
coordinate $\zeta$ for the original curve, defines a smooth local
coordinate system $(\zeta,w)$, with $Z_0$ at the origin.  The matrix
representation of $J$ in this coordinate system is:
  \begin{equation}\label{eq58}
    J(\zeta,w)=\left[\begin{array}{cc}J_{std}&B_1\\0&J_{std}+B_2\end{array}\right],
  \end{equation}
  where the blocks $B_1$, $B_2$ are smooth $2\times2$ matrix functions
  of the coordinates $(\zeta,w)$, satisfying $B_1(\zeta,0)=0$ and
  $B_2(0,0)=0$, so $J(0,0)=J_0$.  By construction, the previously
  given curve ${\bf u}$ in these coordinates is the complex $\zeta$-axis,
  parametrized by $z\mapsto(z,0)$, and the nearby $J$-holomorphic
  curves are parametrized by $z\mapsto(z,c)$, for complex constants
  $c$.  The mapping $z\mapsto(\zeta,w)=(h(z),c)$ is $J$-holomorphic
  for any holomorphic $h$ and constant $c$.

The real entries of the $4\times4$ matrix (\ref{eq58}) (depending on
$\zeta$, $w$) are constrained by the property $J^2=-Id_{\re^4}$, so
they must be of the following form.  It can be assumed that $|b_2|<1$
for $(\zeta,w)$ near $\vec0$:
  \begin{equation}\label{eq60}
    J(\zeta,w)=\left[\begin{array}{cccc}0&-1&a_1&a_2\\1&0&\frac{a_1b_1b_2-a_2b_1^2-a_1b_1-a_2}{b_2-1}&a_1b_2-a_2b_1-a_1\\0&0&b_1&-1+b_2\\0&0&1+\frac{b_1^2+b_2}{1-b_2}&-b_1\end{array}\right].
  \end{equation}
The $-i$ eigenspace can be calculated (\ref{eq67}) and then written in
complex coordinates with smooth complex coefficients $\beta_1$,
$\beta_2$:
\begin{eqnarray}
  T^{0,1}&=&\mbox{span}_{\co}\left\{\frac{\partial}{\partial\bar\zeta},\frac{\partial}{\partial\w}+\beta_1\frac{\partial}{\partial
    w}+\beta_2\frac{\partial}{\partial
    \zeta}\right\}\label{eq59}\\ \beta_1(\zeta,w)&=&\frac{b_2-ib_1}{b_2-2+ib_1}\nonumber\\ \beta_2(\zeta,w)&=&\frac{a_2+i(a_1b_2-a_2b_1-a_1)}{b_2-2+ib_1}\nonumber.
\end{eqnarray}
Conversely, given complex coefficients $\beta_1$, $\beta_2$ in an
expression of the form (\ref{eq59}) with $\left|\beta_1\right|<1$, the
real entries $a_1$, $a_2$, $b_1$, $b_2$ in a complex structure
operator of the form (\ref{eq60}) are uniquely determined by:
\begin{eqnarray*}
  a_1+ia_2&=&\frac{2i(\beta_1\overline{\beta_2}+\beta_2)}{\beta_1\overline{\beta_1}-1}\\
  b_1+ib_2&=&\frac{2i\beta_1(\overline{\beta_1}+1)}{\beta_1\overline{\beta_1}-1}.
\end{eqnarray*}
In terms of $\beta_1$, $\beta_2$, the matrix (\ref{eq60}) for $J(\zeta,w)$ is:
\begin{equation*}
    \left[\begin{array}{cccc}0&-1&\frac{2(\sip(\beta_2)\srp(\beta_1)-\sip(\beta_1)\srp(\beta_2)-\sip(\beta_2))}{|\beta_1|^2-1}&\frac{2(\sip(\beta_2)\sip(\beta_1)+\srp(\beta_2)\srp(\beta_1)+\srp(\beta_2))}{|\beta_1|^2-1}\\1&0&-\frac{2(\sip(\beta_2)\sip(\beta_1)+\srp(\beta_2)\srp(\beta_1)-\srp(\beta_2))}{|\beta_1|^2-1}&\frac{2(\sip(\beta_2)\srp(\beta_1)-\sip(\beta_1)\srp(\beta_2)+\sip(\beta_2))}{|\beta_1|^2-1}\\0&0&-\frac{2\sip(\beta_1)}{|\beta_1|^2-1}&-1+\frac{2(|\beta_1|^2+\srp(\beta_1))}{|\beta_1|^2-1}\\0&0&1-\frac{2(|\beta_1|^2-\srp(\beta_1))}{|\beta_1|^2-1}&\frac{2\sip(\beta_1)}{|\beta_1|^2-1}\end{array}\right].
  \end{equation*}
The integrability condition $[T^{0,1},T^{0,1}]\subseteq T^{0,1}$ is
satisfied when $\frac{\partial\beta_1}{\partial\bar\zeta}$ and
$\frac{\partial\beta_2}{\partial\bar\zeta}$ are both $0$, so
$\beta_1$, $\beta_2$ are holomorphic in $\zeta$.

If $\vec f:\Omega\to\re^4$ is a real variable parametrization, $\vec
f(x,y)=(f^1,f^2,f^3,f^4)$, of a $J$-holomorphic curve in a
neighborhood of $\vec0\in\re^4$, then
\begin{eqnarray}
  d\vec f(x,y)\circ J_{std}&=&J(\vec f(x,y))\circ d\vec
  f(x,y)\label{eq84}\\ \implies\frac{\partial\vec f}{\partial
    y}&=&J(\vec f(x,y))\frac{\partial\vec f}{\partial x}.\label{eq77}
\end{eqnarray}
If the parametric equation is
written in complex form as
\begin{equation}\label{eq80}
  {\bf u}:z\mapsto(\zeta,w)=(h(z),k(z))=(f^1+if^2,f^3+if^4),
\end{equation}
then the $\z$-derivatives of the components are related to the
$z$-derivatives using a $2\times2$ complex matrix ${\bf Q}(\zeta,w)$,
in the following complex nonlinear system of equations:
\begin{equation}\label{eq82}
  \left[\begin{array}{c}h_{\z}\\k_{\z}\end{array}\right]={\bf
    Q}(h,k)\left[\begin{array}{c}\overline{h_{z}}\\\overline{k_{z}}\end{array}\right].
\end{equation}
The calculation deriving ${\bf Q}$ in terms of $J$ is well-known
(\cite{is}, \cite{sikorav}).  However, in this coordinate system, it
is more convenient to express the entries of ${\bf Q}$ in terms of the
coefficients $\beta_1$, $\beta_2$ from the complex eigenvectors
(\ref{eq59}).  

\begin{lem}\label{lem5.1}
  In a coordinate system for a neighborhood of $\re^4$ where $J$ is of
  the form {\rm{(\ref{eq58})}} with complex eigenvectors as in
  {\rm{(\ref{eq59})}}, the matrix ${\bf Q}$ from {\rm{(\ref{eq82})}}
  is of the form
\begin{equation}\label{eq83}
{\bf
  Q}(\zeta,w)=\left[\begin{array}{cc}0&\beta_2(\zeta,w)\\0&\beta_1(\zeta,w)\end{array}\right].
\end{equation}
\end{lem}
\begin{proof}
  The diagonalizing matrix of eigenvectors, its inverse, and the
  diagonalization of $J$ are:
\begin{eqnarray}
    P&=&\left[\begin{array}{cccc}1&1&\beta_2&\bar\beta_2\\i&-i&-i\beta_2&i\bar\beta_2\\0&0&1+\beta_1&1+\bar\beta_1\\0&0&i-i\beta_1&-i+i\bar\beta_1\end{array}\right],\label{eq67}\\
    P^{-1}&=&\frac12\left[\begin{array}{cccc}1&-i&\frac{\bar\beta_2(1-\beta_1)}{\beta_1\bar\beta_1-1}&\frac{i\bar\beta_2(1+\beta_1)}{\beta_1\bar\beta_1-1}\\
    1&i&\frac{\beta_2(1-\bar\beta_1)}{\beta_1\bar\beta_1-1}&\frac{-i\beta_2(1+\bar\beta_1)}{\beta_1\bar\beta_1-1}\\
    0&0&\frac{\bar\beta_1-1}{\beta_1\bar\beta_1-1}&\frac{i(1+\bar\beta_1)}{\beta_1\bar\beta_1-1}\\
    0&0&\frac{\beta_1-1}{\beta_1\bar\beta_1-1}&\frac{-i(1+\beta_1)}{\beta_1\bar\beta_1-1}\end{array}\right],\nonumber\\
    D&=&\left[\begin{array}{cccc}-i&0&0&0\\0&i&0&0\\0&0&-i&0\\0&0&0&i\end{array}\right].\nonumber
\end{eqnarray}
Then, from (\ref{eq77}),
\begin{eqnarray*}
  \frac{\partial\vec f}{\partial
    y}&=&J(\vec f(x,y))\frac{\partial\vec f}{\partial
    x}=PDP^{-1}\frac{\partial\vec f}{\partial x},
\end{eqnarray*}
and this equality of vectors follows:
  \begin{eqnarray*}
    \left[\begin{array}{cccc}-i&0&0&0\\0&i&0&0\\0&0&-i&0\\0&0&0&i\end{array}\right]P^{-1}\left[\begin{array}{c}f^1_x\\f^2_x\\f^3_x\\f^4_x\end{array}\right]&=&P^{-1}\left[\begin{array}{c}f^1_y\\f^2_y\\f^3_y\\f^4_y\end{array}\right]
  \end{eqnarray*}
  \begin{eqnarray}
    &\implies&\left[\begin{array}{c}-if_x^1-f_x^2-i\frac{\bar\beta_2(1-\beta_1)}{\beta_1\bar\beta_1-1}f_x^3+\frac{\bar\beta_2(1+\beta_1)}{\beta_1\bar\beta_1-1}f_x^4\\
    if_x^1-f_x^2+i\frac{\beta_2(1-\bar\beta_1)}{\beta_1\bar\beta_1-1}f_x^3+\frac{\beta_2(1+\bar\beta_1)}{\beta_1\bar\beta_1-1}f_x^4\\
    -i\frac{\bar\beta_1-1}{\beta_1\bar\beta_1-1}f_x^3+\frac{1+\bar\beta_1}{\beta_1\bar\beta_1-1}f_x^4\\
    i\frac{\beta_1-1}{\beta_1\bar\beta_1-1}f_x^3+\frac{1+\beta_1}{\beta_1\bar\beta_1-1}f_x^4\end{array}\right]\label{eq85}\\
  &=&\left[\begin{array}{c}f_y^1-if_y^2+\frac{\bar\beta_2(1-\beta_1)}{\beta_1\bar\beta_1-1}f_y^3+i\frac{\bar\beta_2(1+\beta_1)}{\beta_1\bar\beta_1-1}f_y^4\\
    f_y^1+if_y^2+\frac{\beta_2(1-\bar\beta_1)}{\beta_1\bar\beta_1-1}f_y^3-i\frac{\beta_2(1+\bar\beta_1)}{\beta_1\bar\beta_1-1}f_y^4\\
    \frac{\bar\beta_1-1}{\beta_1\bar\beta_1-1}f_y^3+i\frac{1+\bar\beta_1}{\beta_1\bar\beta_1-1}f_y^4\\
    \frac{\beta_1-1}{\beta_1\bar\beta_1-1}f_y^3-i\frac{1+\beta_1}{\beta_1\bar\beta_1-1}f_y^4\end{array}\right].\label{eq86}
  \end{eqnarray}
  The first and second entries in each vector (\ref{eq85}),
  (\ref{eq86}), are complex conjugates, and the third and fourth
  entries are also conjugates, so for $|\beta_1|\ne1$, the above vector
  equality is equivalent to a system of two complex equations
  (\ref{eq65}), (\ref{eq66}).  Setting the fourth entries of (\ref{eq85}),
  (\ref{eq86}) equal and
  multiplying by $|\beta_1|^2-1$:
\begin{eqnarray}
  i(\beta_1-1)f^3_x+(1+\beta_1)f^4_x&=&(\beta_1-1)f^3_y-i(1+\beta_1)f^4_y\label{eq65}\\
  \implies\frac{\partial}{\partial\z}(f^3+if^4)&=&\beta_1(\vec f(x,y))\cdot\overline{\frac{\partial}{\partial z}(f^3+if^4)}.\nonumber
\end{eqnarray}
 Setting the second entries of (\ref{eq85}), (\ref{eq86}) equal and
 multiplying by $|\beta_1|^2-1$:
\begin{eqnarray}
  &&(\beta_1\bar\beta_1-1)(if^1_x-f^2_x)-i\beta_2(\bar\beta_1-1)f^3_x+\beta_2(1+\bar\beta_1)f^4_x\nonumber\\
  &=&(\beta_1\bar\beta_1-1)(f^1_y+if^2_y)-\beta_2(\bar\beta_1-1)f^3_y-i\beta_2(1+\bar\beta_1)f^4_y\label{eq66}
\end{eqnarray}
\begin{eqnarray}
  \implies\frac{\partial}{\partial\z}(f^1+if^2)&=&\frac1{1-\beta_1\bar\beta_1}\left(-\beta_2\bar\beta_1\frac{\partial}{\partial\z}(f^3+if^4)+\beta_2\overline{\frac{\partial}{\partial z}(f^3+if^4)}\right)\nonumber\\
  &=&\beta_2(\vec f(x,y))\cdot\overline{\frac{\partial}{\partial z}(f^3+if^4)}.\nonumber
\end{eqnarray}
Equation (\ref{eq66}) looks more complicated than (\ref{eq65}), but
there is a simplification using (\ref{eq65}) in the last step.  The
claim that $\bf Q$ as in (\ref{eq82}) is of the form (\ref{eq83})
follows.
\end{proof}
It follows from Lemma \ref{lem5.1} that for ${\bf u}=(h,k)$ as in
(\ref{eq80}), $h$ satisfies a nonlinear, inhomogeneous Cauchy-Riemann
equation
\begin{equation}\label{eq68}
  h_{\z}=\beta_2(h,k)\overline{k_z},
\end{equation}
and $k$ satisfies a Beltrami equation
\begin{equation}\label{eq69}
  k_{\z}=\beta_1(h,k)\overline{k_z}.
\end{equation}

\subsection{The pseudoholomorphically fibered case}

\ 

The results of Sections \ref{sec2} and \ref{sec1} apply to
(\ref{eq68}), so at this point we consider the special case where the
complex structure in normal coordinates satisfies $$\beta_1\equiv0.$$
We also drop the assumption that $\beta_2(\zeta,w)$ is smooth.  The
matrix (\ref{eq60}) for the complex structure operator $J(\zeta,w)$
is:
\begin{eqnarray}
    J(\zeta,w)&=&\left[\begin{array}{cccc}0&-1&a_1&a_2\\1&0&a_2&-a_1\\0&0&0&-1\\0&0&1&0\end{array}\right]\nonumber\\ &=&\left[\begin{array}{cccc}0&-1&2\ip(\beta_2)&-2\rp(\beta_2)\\1&0&-2\rp(\beta_2)&-2\ip(\beta_2)\\0&0&0&-1\\0&0&1&0\end{array}\right].\label{eq74}
\end{eqnarray}
The projection $(\zeta,w)\mapsto w$ is a pseudoholomorphic map
$D_{\vec0,\rho}\times D_{\vec0,\rho}\to D_{\vec0,\rho}$; the fibers
are the $J$-holomorphic curves $(z,c)$ --- this is called the
``pseudoholomorphically fibered'' case by \cite{st1} \S3.

Equations (\ref{eq68}) and (\ref{eq69}), for a parametric map ${\bf
  u}$ as in (\ref{eq80}), become:
\begin{eqnarray}
  h_{\z}&=&\beta_2(h(z),k(z))\overline{k_z},\label{eq73}\\
  k_{\z}&\equiv&0.\nonumber
\end{eqnarray}
So (\ref{eq69}) reduces to the homogeneous Cauchy-Riemann equation,
and $h$ satisfies a nonlinear, inhomogeneous Cauchy-Riemann equation.

The previously stated differentiability assumption in the definition
of $J$-holomorphic curve has been weakened by some authors (e.g.,
\cite{is}) to ${\bf u}\in\cc^0\cap W^{1,2}$, when working with lower
regularity ${\bf u}$ and $J$.  However, for this special case where
${\bf Q}$ is strictly upper-triangular, the $z$-derivative of $h$ does
not appear, and $k$ is already holomorphic by Proposition
\ref{prop2.4}, so as in Section \ref{sec2}, one may consider solutions
of the system without assuming $W^{1,2}$.  More precisely, suppose
${\bf u}=(h,k)$ is a parametric map $\Omega\to\co^2$, where $h$ and
$k$ are continuous, satisfy $(*)$ on $\Omega$, and satisfy the system
(\ref{eq73}) almost everywhere in $\Omega$.  Then $k$ is holomorphic,
and if $\beta_2$ is continuous, then Theorem \ref{lem4.3} and
Corollary \ref{thm4.2} apply to $h$, so the $W^{1,2}$ property follows
as a conclusion.  Further, it follows immediately from (\ref{eq73})
and Liouville's Theorem that for any $\beta_2$, if ${\bf
u}:\co\to\co^2$ has bounded image then it is constant.

If $\beta_2(\zeta,w)$ has a factorization of the separable form
$f(\zeta)\frac{\partial v(w)}{\partial \w}$, then using the chain
rule,
$$\frac{\partial}{\partial\z}(v(k(z)))=v_w(k(z))k_{\z}+v_{\w}(k(z))\overline{k_z},$$
(\ref{eq73}) can be re-written as:
\begin{eqnarray}
  h_{\z}&=&f(h(z))v_{\w}(k(z))\overline{k_z}=f(h(z))\frac{\partial}{\partial\z}(v(k(z))).\label{eq79}
\end{eqnarray}

\begin{example}\label{ex4.2}
  Consider the function $\beta_2(\zeta,w)=\zeta^2\bar w$ and the
  corresponding almost complex structure $J$ (\ref{eq74}) on $\co^2$.
  The system (\ref{eq73}) for ${\bf u}:\Omega\to\co^2$ becomes:
  \begin{eqnarray}
    h_{\z}&=&\beta_2(h(z),k(z))\overline{k_z}=h^2\overline{k}\overline{k_z}\label{eq19}\\
    k_{\z}&\equiv&0,\nonumber
  \end{eqnarray}
  so a continuous ${\bf u}$ satisfying $(*)$ and (\ref{eq19}) almost
  everywhere on $\Omega$ must be real analytic by Corollary
  \ref{thm4.2}.  In fact, this $J$ defines an integrable almost
  complex structure on $\co^2$, so we do not expect the local
  qualitative behavior of $J$-holomorphic curves to be different from
  standard holomorphic curves.  However, the results of Section
  \ref{sec1} allow us to explicitly compute local parametric formulas
  for all the $J$-holomorphic curves in this coordinate system.  As in
  (\ref{eq79}), Equation (\ref{eq19}) can be re-written
  $$h_{\z}=h^2\overline{\frac{\partial}{\partial
    z}(\frac12k^2)}=h^2\frac{\partial}{\partial\z}\left(\overline{\frac12k^2}\right),$$
    so this is a separable equation, to which Theorem \ref{thm2.4} and
    Theorem \ref{thm2.9} apply, with
    $G(z)=\overline{\frac12(k(z))^2}$, $f(w)=w^2$, and
    $F(w)=-\frac1w$.  If $z_0$ is any point in $\Omega_1$ with
    $h(z_0)=\zeta_0\ne0$, then by the constructions in the Proofs of
    Lemma \ref{thm2.2} and Theorem \ref{thm2.4}, for $z$ near $z_0$
  $$h(z)=\frac{-1}{\overline{\frac12(k(z))^2}-\overline{\frac12(k(z_0))^2}+C(z)-\frac1{\zeta_0}},$$
  for some holomorphic $C$ with $C(z_0)=0$.

  If $h(z_0)=0$ (so ${\bf u}=(h,k)$ meets the $w$-axis), then $h$ is
  either $\equiv0$ (${\bf u}=(0,k(z))$ is $J$-holomorphic), or has the
  following form, by Theorem \ref{thm2.9}:
  \begin{equation}\label{eq22}
    h(z)=\frac{-(z-z_0)^m}{\phi(z)+(z-z_0)^m\overline{\frac12(k(z))^2}}.
  \end{equation}
  In this case, choosing any $m\ge1$, holomorphic $\phi$ with
  $\phi(z_0)\ne0$, and holomorphic $k(z)$ gives an example of a
  solution $h$.
\end{example}
\begin{example}\label{ex4.3}
  Consider the function
  $$\beta_2(\zeta,w)=\frac{\partial V}{\partial\w}(w),$$ where $V$ is
  the function constructed in Example \ref{ex4.4}, depending on $w$.
  The corresponding almost complex structure $J$ (\ref{eq74}) is
  continuous on $\co^2$, and equal to the standard complex structure
  $J_0$ outside a neighborhood of the origin.  The system (\ref{eq73})
  for ${\bf u}:\Omega\to\co^2$ becomes, as in (\ref{eq79}):
  \begin{eqnarray}
    h_{\z}&=&\beta_2(h(z),k(z))\overline{k_z}=\frac{\partial V}{\partial\w}(k(z))\overline{k_z}=\frac{\partial}{\partial\z}(V(k(z)))\label{eq76}\\
    k_{\z}&\equiv&0.\nonumber
  \end{eqnarray}
  If ${\bf u}:\Omega\to\co^2$ is continuous, satisfies $(*)$, and
  satisfies (\ref{eq76}) almost everywhere on $\Omega$, then $k$ is
  holomorphic, and by Theorem \ref{lem4.3}, for any $R\Subset\Omega$
  and $0<\alpha<1$, $h|_R\in W^{1,2}(R)\cap\cc^{0,\alpha}(R)$.  By
  Lemma \ref{lem2.3},
  \begin{equation}\nonumber
    h(z)=V(k(z))+C(z)
  \end{equation}
  for some holomorphic function $C$.  One example of such a solution
  ${\bf u}:\co\to\co^2$ is $(h,k)=(V(z),z)$.

  This Example shows that there exists a continuous almost complex
  structure $J$, admitting a differentiable $J$-holomorphic curve
  ${\bf u}=(h,k)$ which is a solution of the matrix equation
  (\ref{eq82}) such that both LHS and RHS of (\ref{eq82}) are defined
  everywhere and continuous (after the matrix multiplication), but
  ${\bf u}$ is not $\cc^1$ because the LHS and RHS of (\ref{eq77}) are
  not locally bounded.
\end{example}

\subsubsection*{Acknowledgments}
  The authors acknowledge the helpful comments of an anonymous
  referee, who suggested essentially all of Subsection
  \ref{subsec4.2}.  The first author was supported in part by a 2015
  sabbatical semester at IPFW.  The third author was supported in part
  by National Science Foundation research grant DMS-1265330.

\end{document}